\documentclass[10pt]{amsart}
\usepackage{tikz}
\usepackage{multicol}
\usepackage{amsmath,amssymb,amsthm,amsfonts}
\usepackage{mathrsfs,eufrak,euscript}
\usepackage{tikz}
\usepackage{enumerate}
\theoremstyle{definition}
\newtheorem{thm}[subsection]{Theorem}
\newtheorem{prop}[subsection]{Proposition}
\newtheorem{lemma}[subsection]{Lemma}
\newtheorem{eg}[subsection]{Example}
\newtheorem{eg's}[subsection]{Examples}

\newtheorem{defn}[subsection]{Definition}
\newtheorem{rmk}[subsection]{Remark}

\newtheorem{note}[subsection]{Note}
\newtheorem{cor}[subsection]{Corollary}

\DeclareRobustCommand{\rchi}{{\mathpalette\irchi\relax}}
\newcommand{\irchi}[2]{\raisebox{\depth}{$#1\chi$}} 
\title{A note on convexity of Sections of quaternionic numerical range }
\author{P. Santhosh Kumar }
\address{P. Santhosh Kumar \\ 
	Statistics and Mathematics Unit, \\
	Indian Statistical Institute Bangalore, \\
	8th Mile, Mysore Road,
	RVCE Post
	Bangalore 560 059.}
\email{santhosh.uohmath@gmail.com}
\begin{document}
\maketitle
\begin{abstract}
	The quaternionic numerical range of matrices over the ring of  quaternions is not necessarily  convex. We prove Toeplitz-Hausdorff like theorem, that is, for any given quaternionic matrix every section of its quaternionic numerical range is convex. We provide some additional equivalent conditions for the quaternionic numerical range of matrices over quaternions to be convex and prove some numerical radius inequalities. 
\end{abstract}
\section{Introduction and Preliminaries}
In case of matrices over the field of complex numbers, it is  well known from the Toeplitz-Hausdroff theorem that the complex numerical range is a convex subset of $\mathbb{C}$ (see \cite{GustafsonRao} for details). Whereas, the quaternionic numerical range of matrices with quaternion entries need not be convex (see Section 2). It was J. E. Jamison \cite{Jamison}, who proposed the problem to characterize the class of linear operators on quaternionic Hilbert space with  convex numerical range. Besides this problem, several authors studied  the properties of the intersection of quaternionic numerical range  with $\mathbb{R}$ and $\mathbb{C}$  (see \cite{Au-Yeung1, Au-Yeung2, Jamison, So1, So2}).

 It was shown by So and Thompson \cite{So1, So2} that the intersection of quaternionic numerical range and the closed upper half plane is convex, but this proof is very long. In fact this intersection of a given quaternion matrix is not a complex numerical range (see \cite{Thompson} for details).  A question which is arised,  is there a  short  and conceptual proof for the result by So and Thompson? \cite[Question 1]{Zhang}. In this article, we address this question by proving Topelitz-Hausdorff like theorem. Our approach is different from the method followed in \cite{So2}. We study the geometry of the sections of quaternionic numerical range in case of $2\times 2$ quaternionic matrices, then establish the result for $n\times n \; (n \in \mathbb{N})$ quaternionic matrices by employing the technique given in \cite{GustafsonRao}. We add some equivalent conditions to the result proved by Au-Yeung \cite{Au-Yeung2}, for the quaternionic numerical range of matrices over quaternions to be convex. Also we prove some numerical range inequalities which are analogous to the classical results.

In the first section we recall some basic definitions and results which are useful for later sections. In the second  section we prove equivalent condition for the convexity of quaternionic numerical range and prove Toeplitz-Hausdorff like theorem.  In the final section we prove some inequalities related to numerical radius of quaternionic matrices.

\paragraph{\bf Quaternions}
Let $\mathbb{H}$ be the set of all elements, called quaternions, of the form $q = q_{0}+q_{1}i+q_{2}j+q_{3}k$, where $i,j,k$ are quaternion units satisfying:
\begin{equation}\label{Equation:multiplication}
i\cdot j = -j \cdot i, \; \; j\cdot k = -k \cdot j,\; \; k\cdot i = -i \cdot k \; \; \text{and}\;\; i\cdot j \cdot k = -1.
\end{equation}
Then $\mathbb{H}$ is a non-commutative division ring with the addition defined same as in  $\mathbb{C}$ and multiplication given by Equation $(\ref{Equation:multiplication})$.  For a given $q \in \mathbb{H}$, we define the real part, re$(q):= q_{0}$ and the imaginary part, im$(q):= q_{1}i+q_{2}j+q_{3}k$. The conjugate and the modulus of $q$ respectively given by 
\begin{equation*}
\overline{q}= q_{0} - ( q_{1}i+q_{2}j+q_{3}k)\; , \; |q|= \sqrt{q_{0}^{2}+q_{1}^{2}+q_{2}^{2}+q_{3}^{2}}.
\end{equation*}
The set of all  imaginary unit quaternions,  denoted by $\mathbb{S}$, is defined as
\begin{equation*}
\mathbb{S}: = \{q \in \mathbb{H}\;:\; \overline{q} = -q \; \&\;|q|=1\} = \{q\in \mathbb{H}\; :\; q^{2}=-1\}.
\end{equation*}
For each $m \in \mathbb{S}$, a slice $\mathbb{C}_{m}$ is defined as 
\begin{equation*}
\mathbb{C}_{m}: = \{\alpha+m \beta \in \mathbb{H}:\; \alpha, \beta \in \mathbb{R}\}. 
\end{equation*}
It is a real subalgebra of $\mathbb{H}$ generated by $\{1, m\}$. It is isomorphic as a field to $\mathbb{C}$, for each $m \in\mathbb{S}$. The upper half plane of $\mathbb{C}_{m}$ is, $\mathbb{C}_{m}^{+} = \{\alpha + m \beta\; |\; \alpha \in \mathbb{R}, \beta >0\} $. In fact $\mathbb{C}_{m} \cap \mathbb{C}_{n} = \mathbb{R}$, for $m \neq \pm n$ and $\mathbb{H}= \bigcup\limits_{m \in \mathbb{S}}\mathbb{C}_{m}$. For every $p, q \in \mathbb{H}$, define
\begin{equation*}
p \sim q \;\; \text {if and only if} \;\; p = s^{-1}qs, \; \text{for some}\; s \in \mathbb{H}\setminus\{0\}.
\end{equation*}
It is an equivalence relation on $\mathbb{H}$. The equivalence class of $q$, denoted by $[q]$, is given by 
\begin{equation*}
[q] = \Big\{p \in \mathbb{H}:\; \text{re}(q)=\text{re}(p), \; |\text{im}(q)|= |\text{im}(p)|\Big\}.
\end{equation*}
\begin{note} For every  $q\in \mathbb{H}$, we observe the following: 
	\begin{enumerate}
		\item   If $q = q_{0} +q_{1}i+q_{2}j+q_{3}k $ and $m \in \mathbb{S}$, then  $q \sim q_{0} \pm m \sqrt{q_{1}^{2}+q_{2}^{2}+q_{3}^{2}}$.    Moreover, 
		\begin{equation*}
		[q] \cap \mathbb{C}_{m} = \Big\{ q_{0} \pm m \sqrt{q_{1}^{2}+q_{2}^{2}+q_{3}^{2}} \Big\}, \; \text{for every}\; m \in \mathbb{S}.
		\end{equation*}
		\item There exist a unique pair  $(z_{1}, z_{2}) \in \mathbb{C}^{2}$  such that $q = z_{1}+z_{2}\cdot \; j$, where 
		\begin{equation*}
		z_{1}= \frac{1}{2}\Big[(q+\overline{q})- (qi+\overline{qi})i\Big], \;\; z_{2}= -\frac{1}{2}\Big[(qj+\overline{qj})+ (qk+\overline{qk})i\Big].
		\end{equation*} 
	\end{enumerate}
\end{note}
\begin{defn}\label{Definition:circularization}
	Let ${S}$ be a non-empty subset of  $\mathbb{C} $. Then the {\it circularization} of $S$ in $\mathbb{H}$, denoted by $\Omega_{S}$, is defined  as
	\begin{equation*}
	\Omega_{S} = \big\{\alpha + \beta m \in \mathbb{H}\; :\;  \alpha, \beta \in \mathbb{R}, \; \alpha + i \beta  \in S,\; m \in \mathbb{S} \big\}.
	\end{equation*}
	Equivalently,  $\Omega_{S} = \bigcup\limits_{z \in S} [z]$.
	
	\noindent A subset $\mathcal{K}$ of $\mathbb{H}$ is said to be {\it circular (or axially symmetric)}, if $\mathcal{K} = \Omega_{S}$ for some $S$.
\end{defn}
\begin{eg} Ring of quaternions $\mathbb{H}$ is circular since 
	$\Omega_{\mathbb{C}} = \mathbb{H}$ by $(1)$ of Note 1.1. In particular,
	\begin{equation*}
	[q] = \Omega_{\big\{\text{re}(q) + i \; | \text{im}(q)|\big\} }
	\end{equation*} 
	is circular, for every $q \in \mathbb{H}$.
\end{eg}
We recall the notion of innerproduct. Let $\mathbb{F}$ be either $\mathbb{C}$ or $\mathbb{H}$ and  $\langle \cdot , \cdot \rangle_{\mathbb{F}}$ denote the inner product on $\mathbb{F}^{n}$ (for $n \in \mathbb{N}$), defined by 

\begin{equation*}
\big\langle (p_{l})_{l=1}^{n}, (q_{l})_{l=1}^{n} \big\rangle_{\mathbb{F}}:= \sum\limits_{l=1}^{n}\overline{p}_{l} q_{l}, \; \text{for all}\; (p_{l})_{l=1}^{n},(q_{l})_{l=1}^{n}\in \mathbb{F}^{n}. 
\end{equation*}
We adopt the convention that the innerproduct is conjugate linear in the first variable and linear in the second variable.   In particular, $(\mathbb{H}^{n}, \langle \cdot, \cdot \rangle_{\mathbb{H}})$ is a right quaternionic Hilbert space (for definition see \cite{Ghiloni, Santhosh} and references therein).
\vspace{0.5cm}

\paragraph{\bf Role of complex matrices}
Let us  denote the class of all $n\times n \;(\text{for}\; n\in\mathbb{N})$ matrices over  $\mathbb{C}$  and  $\mathbb{H}$  by $M_{n}(\mathbb{C})$ and $M_{n}(\mathbb{H})$ respectively.  The conjugate of matrix  $B = [\; b_{rs}\; ]_{n\times n} \in M_{n}(\mathbb{C})$ is defined by $\overline{B} = [\; \overline{b}_{rs}\; ]_{n\times n}$. 

Let $A = [\;q_{rs}\;]_{n\times n}\in M_{n}(\mathbb{H})$. Since  $q_{rs} = a_{rs} + b_{rs}\cdot j$ for some $a_{rs}, b_{rs} \in \mathbb{C}$ by $(2)$ of Note 1.1, then $A_{1}:= [\; a_{rs}\; ]_{n\times n},\; A_{2}:= [\; b_{rs}\; ]_{n\times n}\in M_{n}(\mathbb{C})$ and  $ A=A_{1}+A_{2}\cdot j$. Define 
\begin{equation*}
\rchi_{A}:= \begin{bmatrix}
A_{1} & A_{2}\\-\overline{A}_{2}& \overline{A}_{1}
\end{bmatrix}_{2n\times 2n} \in M_{2n}(\mathbb{C}),
\end{equation*} 
then the map $\xi \colon M_{n}(\mathbb{H})\to M_{2n}(\mathbb{C})$ defined by  $
\xi(A) = \rchi_{A}, \; \text{for all}\; A \in M_{n}(\mathbb{H})$ is an injective real algebra homomorphism.
It is clear from \cite[Proposition 5.4, 5.5]{Santhosh} that  $\|A\| = \|\rchi_{A}\|$, where $\|\cdot \|$ denotes operator norm in the respective algebra.

In particular, if $q = q_{0}+q_{1}i+q_{2}j+q_{3}k \in \mathbb{H}\setminus \{0\}$, then 
\begin{equation*}
\rchi_{q}= \begin{bmatrix}
q_{0}+q_{1}i & q_{2}+q_{3}i\\
-q_{2}+q_{3}i& q_{0}-q_{1}i
\end{bmatrix}_{2\times2} \in M_{2}(\mathbb{C})
\end{equation*}
has eigenvalues $q_{0} \pm  i \sqrt{q_{1}^{2}+q_{2}^{2}+q_{3}^{2}}$ and $det(\rchi_{q}) = |q|^{2}$.

The properties of $A$ is inherited from $\rchi_{A}$ and vice versa. See the following theorem.
\begin{thm}\cite[Theorem 5.6]{Santhosh}\label{Theorem: properties}
	Let $A \in M_{n}(\mathbb{H})$. Then 
	\begin{enumerate}
		\item $\rchi_{A}^{\ast} = \rchi_{A^{\ast}}$.
		\item $\rchi_{A}$ is self-adjoint if and only if $A$ is self-adjoint.
		\item $\rchi_{A}$ is positive if and only if $A$ is positive.
		\item $\rchi_{A}$ is normal if and only if $A$ is normal.
		\item $\rchi_{A}$ is unitary if and only if $A$ is unitary.
		
	\end{enumerate}
\end{thm}
Note that $\rchi_{A}$ is invertible if and only if $A$ is invertible. The inverse of $\rchi_{A}$ is given by $(\rchi_{A})^{-1} = \rchi_{A^{-1}}$.
\vspace{0.5cm}

\paragraph{\bf Numerical range} We define the quaternionic numerical range and the sections of quaternionic numerical range  as follows:
\begin{defn}
	Let $A \in M_{n}(\mathbb{H})$. Then
	\begin{enumerate}
		\item The quaternionic numerical range of $A$, denoted by $W_{\mathbb{H}}(A)$, defined as
		\begin{equation*}
		W_{\mathbb{H}}(A):= \big\{\langle X, AX \rangle_{\mathbb{H}}: X\in S_{\mathbb{H}^{n}} \big\},
		\end{equation*}
		where $S_{\mathbb{H}^{n}}: = \big\{X \in \mathbb{H}^{n}: \|X\|=1\big\}$.
		\item The quaternionic numerical radius of $A$, denoted by ${\mathop{w}}_{\mathbb{H}}(A)$, defined as
		\begin{equation*}
		{\mathop{w}}_{\mathbb{H}}(A):= \sup\big\{|q|: q \in W_{\mathbb{H}}(A)\big\}.
		\end{equation*}
		\item For each slice $\mathbb{C}_{m}$ of $\mathbb{H} \; (m \in \mathbb{S})$, we call $W_{\mathbb{H}}(A)\cap \mathbb{C}_{m}^{+}$ as a $\mathbb{C}_{m}$- section of $W_{\mathbb{H}}(A)$.  In particular, we denote the complex section of $W_{\mathbb{H}}(A)$ by $W_{\mathbb{H}}^{+}(A)$ i.e., 
		\begin{equation*}
		W_{\mathbb{H}}^{+}(A):= W_{\mathbb{H}}(A)\cap \mathbb{C}^{+} ,
		\end{equation*}
		where $\mathbb{C}^{+}= \{\alpha + i \beta :\; \alpha \in \mathbb{R},\; \beta \geq 0\}$.
	\end{enumerate}  
\end{defn}
Note that $ {W_{\mathbb{H}}(A)} \cap \mathbb{C}_{m}^{+} \neq \emptyset $ for each $m \in \mathbb{S}$ (see Lemma \ref{Lemma:circular}).
\begin{defn}\cite{Zhang}\label{Definition:complexpart}Let $A \in M_{n}(\mathbb{H})$. Then the projection of $W_{\mathbb{H}}(A)$ onto the complex plane is denoted by $W_{\mathbb{H}}(A :\mathbb{C})$ and it is defined by
	\begin{equation*}
	W_{\mathbb{H}}(A :\mathbb{C}) = \{co(q); \; q \in W_{\mathbb{H}}(A)\},
	\end{equation*}
	where $co(q) = q_{0}+q_{1}i$, for $q = q_{0}+q_{1}i+q_{2}j+q_{3}k \in \mathbb{H}$.
\end{defn}
\vspace{0.5cm}
\paragraph{\bf Spherical spectrum} Unilike in the case of complex matrices, the left and the right eigenvalues of quaternionic matrices are to be treated  differently. For physical applications, we consider right eigenvalues of quaternionic matrices.  The spherical spectrum of a quaternionic matrix is the collection of all right eigenvalues.  
\begin{defn} \label{Definition:sphericalspectrum}
	Let $A \in M_{n}(\mathbb{H})$. Define  $\Delta_{q}(A):= A^{2}-2 \; \text{re}(q)A+ |q|^{2}I$ for $q \in \mathbb{H}$, then  the {\it spherical spectrum} of $A$, denoted by $\sigma_{S}(A)$, defined as
	\begin{equation*}
	\sigma_{S}(A) = \big\{q\in \mathbb{H}:\; N(\Delta_{q}(A)) \neq \{0\} \big\}.
	\end{equation*} 
	Here $N(\Delta_{q}(A)):= \{X \in \mathbb{H}^{n}\; :\; \Delta_{q}(A)X = 0 \}$, the null space of $\Delta_{q}(A)$.  
\end{defn}
\begin{note}\label{Note: KeyNote}The following key observations are useful to determine the spherical spectrum of any given quaternionic matrx:    \begin{enumerate}
		\item If $\Delta_{q}(A)X = 0$ for some $q \in \mathbb{H}$ and $X\in \mathbb{H}^{n} \setminus \{0\}$, then
		\begin{equation*}
		A(AX-Xq) - (AX-Xq)\overline{q} = 0.
		\end{equation*}
		Suppose that $AX-Xq=0$, then $q$ is a right eigenvalue of $A$. Otherwise, $AY = Y\overline{q}$, where $Y := AX-Xq \neq 0$ i.e., $\overline{q}$ is a right eigenvalue of $A$.
		\item If $q \in \sigma_{S}(A)$, then 
		\begin{align*}
		\Delta_{s^{-1}qs}(A) &= A^{2} - 2\; \text{re}(s^{-1}qs)A+|q|^{2}I\\
		&= A^{2}-2\; \text{re}(q)A+ |q|^{2}I\\
		&= \Delta_{q}(A).
		\end{align*}    
		This implies that $q \in \sigma_{S}(A)$ if and only if $[q] \in \sigma_{S}(A)$. Equivalently, $\sigma_{S}(A)$ is circular.
		\item If $q \in \sigma_{S}(A)$, then by above observations  $z:= \text{re}(q)+ i\;| \text{im}(q)| \in \sigma_{S}(A)$ and there exist a $X = X_{1}+X_{2}\cdot j \in \mathbb{H}^{n}\setminus \{0\}$, where $X_{1}, X_{2} \in \mathbb{C}^{n}$ such that $AX = Xz$. It implies that 
		\begin{equation*}
		\rchi_{A}\begin{bmatrix}
		X_{1}\\
		-\overline{X}_{2}
		\end{bmatrix} = \begin{bmatrix}
		X_{1}\\
		-\overline{X}_{2}
		\end{bmatrix} z
		\end{equation*}
		i.e.,  $z$ is an eigenvalue of $\rchi_{A}$ and vice versa.  Therefore, it is sufficient to know the  eigenvalues of $\rchi_{A}$ to determine right eigenvalues of $A$.  Since $\sigma_{S}(A)$ is circular, we conclude that $\sigma_{S}(A) = \Omega_{\sigma(\rchi_{A})}$, where $\sigma(\rchi_{A})$ is the set of all complex eigenvalues of  $\rchi_{A}$. 
	\end{enumerate}
	Since $\rchi_{A}$ is similar to $\overline{\rchi_{A}}$, then the non-real eigenvalues of $\rchi_{A}$ occur in conjugate pairs with the same multiplicity and the real eigenvalues occur an even number of times.  All the eigenvalues of $\rchi_{A}$ with non negative imaginary part, whcih  are called {\bf standard eigenvalues} of $A$ (see \cite{Zhang} for detials), are enough to know the spherical spectrum of $A$. That is, 
	\begin{equation}\label{Equation: Standardeigenvalues}
	\sigma_{S}(A) = \Omega_{\sigma(\rchi_{A})} = \Omega_{\sigma(\rchi_{A}) \cap \mathbb{C}^{+}}.
	\end{equation}
\end{note}
\begin{eg}
	If  $A = \begin{bmatrix}
	j & 0\\
	0 & -j
	\end{bmatrix}$, then $\rchi_{A} = \begin{bmatrix}
	0&0&1&0\\
	0&0&0&-1\\
	-1&0&0&0\\
	0&1&0&0
	\end{bmatrix}$.
	Here ${\bf \pm \it{i}}$ are the eigenvalues of $\rchi_{A}$, then $i$ is the standard eigenvalue of  $A$. Then by Equation (\ref{Equation: Standardeigenvalues}), we have $\sigma_{S}(A) = \Omega_{\{i\}} = \mathbb{S}$.
\end{eg}
\section{Convexity of numerical range sections}

In general quaternionic numerical range of matrices over the ring of quaternions  is not necessarily convex. For example,
\begin{equation*}
A = \begin{bmatrix}
k &0&0\\
0&1&0\\
0&0&1
\end{bmatrix}_{3\times 3} \in M_{3}(\mathbb{H})
\end{equation*} 
has  $k, -k \in W_{\mathbb{H}}(A)$, but $0 = \frac{k}{2} - \frac{k}{2} \notin W_{\mathbb{H}}(A)$.
To see this, assume that there  is a  $X: = \begin{bmatrix}
x&y&z
\end{bmatrix}^{T} \in S_{\mathbb{H}^{3}}$ such that 
\begin{equation*}
0 = \big\langle X , A X \big\rangle_{\mathbb{H}} = \overline{x}kx+ |y|^{2}+|z|^{2}
\end{equation*}
i.e., $|y|^{2}+|z|^{2} = -\overline{x}kx $. This is contradiction since $\overline{\overline{x}kx} = - \overline{x}kx $ and $|y|^{2}+|z|^{2}$ is real.  It shows that $W_{\mathbb{H}}(A)$ is not convex.

The next choice  is to investigate the convexity of sections of quaternionic numerical range, because for any give matrix $A \in M_{n}(\mathbb{H})$, the circularization of every section of quaternionic numerical range is $W_{\mathbb{H}}(A)$. 

Though the complex section of a quaternionic numerical range can not be realised as a complex numerical range of some complex matrix (see \cite{Thompson} for details), we prove that the complex projection of quaternionic numerical range of $A \in M_{n}(\mathbb{H})$ is same as the complex numerical range of $\rchi_{A}$. 
\begin{prop}\label{Proposition:projection}
	Let $A\in M_{n}(\mathbb{H})$. Then $W_{\mathbb{H}}(A:\mathbb{C}) = W_{\mathbb{C}}(\rchi_{A})$.
\end{prop}
\begin{proof} Given $A \in M_{n}(\mathbb{H})$ is decomposed as $A = A_{1}+A_{2}\cdot j$, where $A_{1}, A_{2} \in M_{n}(\mathbb{C})$. If $z \in W_{\mathbb{H}}(A:\mathbb{C})$, then there is a $X:= X_{1}+X_{2}\cdot j \in S_{\mathbb{H}^{n}}$, where $X_{1},X_{2}\in \mathbb{C}^{n}$ such that  
	\begin{align*}
	z & =co \big(\langle X, AX\rangle_{\mathbb{H}}\big) \\
	&= co\big(\langle X_{1}+X_{2}\cdot j, (A_{1}+A_{2}\cdot j)(X_{1}+X_{2}\cdot j)\rangle_{\mathbb{H}}\big)\\
	&= \langle X_{1}, A_{1}X_{1}-A_{2}\overline{X}_{2}\rangle_{\mathbb{H}} + \overline{j}\langle X_{2}, A_{1}X_{2}+A_{2}\overline{X}_{1} \rangle_{\mathbb{H}}\cdot j\\
	&= \langle X_{1}, A_{1}X_{1}-A_{2}\overline{X}_{2}\rangle_{\mathbb{C}} + \overline{\langle X_{2}, A_{1}X_{2}+A_{2}\overline{X}_{1}\rangle}_{\mathbb{C}}  \\
	&= \langle X_{1}, A_{1}X_{1}-A_{2}\overline{X}_{2}\rangle_{\mathbb{C}} + \overline{(-\overline{X}_{2})} (-\overline{A}_{1}\overline{X_{2}}- \overline{A}_{2}X_{1})\\
	&= \langle X_{1}, A_{1}X_{1}-A_{2}\overline{X}_{2}\rangle_{\mathbb{C}}+ \langle -\overline{X}_{2}, -\overline{A}_{1}\overline{X}_{2}- \overline{A}_{2}X_{1}\rangle_{\mathbb{C}}\\
	&= \Big\langle \begin{bmatrix}
	X_{1}\\
	-\overline{X}_{2}
	\end{bmatrix}, \begin{bmatrix}
	A_{1}& A_{2}\\
	-\overline{A}_{2} & \overline{A}_{1}
	\end{bmatrix}\begin{bmatrix}
	X_{1}\\
	-\overline{X}_{2}
	\end{bmatrix} \Big\rangle_{\mathbb{C}^{n}\oplus \mathbb{C}^{n}}.
	\end{align*}
	Since  $X_{1}+ X_{2}\cdot j \in S_{\mathbb{H}^{n}}$, we have $\begin{bmatrix}
	X_{1}\\
	-\overline{X}_{2}
	\end{bmatrix}\in S_{\mathbb{C}^{n}\oplus \mathbb{C}^{n}}$ . This implies that $z \in W_{\mathbb{C}}(\rchi_{A})$.  Conversly, assume that $ {\lambda} \in W_{\mathbb{C}}(\rchi_{A})$, then there is a $\begin{bmatrix}
	Y_{1}\\
	Y_{2}
	\end{bmatrix} \in S_{\mathbb{C}^{n}\oplus \mathbb{C}^{n}}$ such that 
	\begin{align*}
	\lambda &=  \Big\langle \begin{bmatrix}
	Y_{1}\\
	Y_{2}
	\end{bmatrix}, \begin{bmatrix}
	A_{1}& A_{2}\\
	-\overline{A}_{2} & \overline{A}_{1}
	\end{bmatrix} \begin{bmatrix}
	Y_{1}\\
	Y_{2}
	\end{bmatrix} \Big\rangle_{\mathbb{C}^{n}\oplus \mathbb{C}^{n}} \\
	&= \langle Y_{1}, A_{1}Y_{1} + A_{2}Y_{2}\rangle_{\mathbb{C}} + \langle Y_{2}, -\overline{A}_{2}Y_{1} + \overline{A}_{1}Y_{2}\rangle_{\mathbb{C}} \\
	&= co \big( \langle Y, AY\rangle_{\mathbb{H}} \big),
	\end{align*}
	where $Y: = Y_{1} -\overline{Y}_{2}\cdot j$ and  $\|Y\| = 1$. This shows that $\lambda \in W_{\mathbb{H}}(A:\mathbb{C})$. Therefore   $W_{\mathbb{H}}(A:\mathbb{C}) = W_{\mathbb{C}}(\rchi_{A})$.
\end{proof}

\begin{note}
	Since  $W_{\mathbb{C}}(\rchi_{A})$ is convex by Toeplitz-Housdroff theorem \cite[Theorem 1.1-2]{GustafsonRao}, then by Proposition \ref{Proposition:projection} , $W_{\mathbb{H}}(A:\mathbb{C})$ is convex. In particular, for a  self-adjoint matrix $A \in M_{n}(\mathbb{H})$,  
	\begin{equation*}
	W_{\mathbb{H}}(A) = W_{\mathbb{H}}(A:\mathbb{C}) = W_{\mathbb{C}}(\rchi_{A})
	\end{equation*}
	is a convex subset of $\mathbb{R}$.
\end{note}
\begin{lemma}\label{Lemma:circular}
	Let $A \in M_{n}(\mathbb{H})$. Then $W_{\mathbb{H}}(A)$ is circular (or axially symmetric).
	
\end{lemma}
\begin{proof}
	If $q \in W_{\mathbb{H}}(A)$ then  $q = \langle X, AX\rangle_{\mathbb{H}}$ for some $X\in S_{\mathbb{H}^{n}}$. For every $ s \in \mathbb{H}\setminus \{0\}$, we have 
	\begin{equation*}
	s^{-1}qs = s^{-1} \cdot \langle X, AX\rangle_{\mathbb{H}}\cdot s = \frac{\overline{s}}{|s|}\cdot \langle X, AX\rangle_{\mathbb{H}}\cdot \frac{{s}}{|s|} = \big\langle X \cdot \frac{s}{|s|} , A (X \cdot \frac{s}{|s|}) \big\rangle_{\mathbb{H}} .
	\end{equation*} 
	Take $X_{s}:= X\cdot \frac{{s}}{|s|}$, then  $\|X_{s}\| = 1$ and $s^{-1}qs = \langle X_{s}, AX_{s}\rangle_{\mathbb{H}} \in W_{\mathbb{H}}(A) $.  That is, $[q] \in W_{\mathbb{H}}(A)$. This shows that $W_{\mathbb{H}}^{+}(A)= \{\text{re}(q) + i \;|\text{im(q)}|:\; q \in W_{\mathbb{H}}(A)\}$ is a non-empty subset of $\mathbb{C}$ and $W_{\mathbb{H}}(A) = \Omega_{W_{\mathbb{H}}^{+}(A)}$. Hence $W_{\mathbb{H}}(A)$ is circular. 
\end{proof}
By Lemma \ref{Lemma:circular}, every section of quaternionic numerical range of $A$ i.e., $W_{\mathbb{H}}(A) \cap \mathbb{C}_{m}^{+}$ is non-empty and $W_{\mathbb{H}}(A)= \Omega_{W_{\mathbb{H}}(A) \cap \mathbb{C}_{m}^{+}}$, for each $m \in \mathbb{S}$.
\begin{rmk}\label{Remark: Onlysubset}
	Let $q \in W_{\mathbb{H}}(A)$. Then by Proposition \ref{Proposition:projection} and Lemma \ref{Lemma:circular}, we have 
	\begin{equation*}
	z_{\pm}:= \text{re}(q) \pm i \;|\text{im}(q)| \in {W_{\mathbb{H}}({A})}\cap \mathbb{C} \subseteq W_{\mathbb{H}}(A:\mathbb{C}) = W_{\mathbb{C}}(\rchi_{A}).
	\end{equation*}
	This implies that $q \in [z_{\pm}] \subseteq  \Omega_{W_{\mathbb{C}(\rchi_{A})}}$. Therefore,
	\begin{equation}\label{Equation:propersubset}
	W_{\mathbb{H}}(A) \subseteq \Omega_{W_{\mathbb{C}}(\rchi_{A})}.
	\end{equation}
\end{rmk}
Note that the equality may not hold in Equation (\ref{Equation:propersubset}). See the following example. 
\begin{eg}\label{Example: NeedNot}
	Let $A = \begin{bmatrix}
	j
	\end{bmatrix}_{1\times 1}  \in  \mathbb{H}$, then
	$\rchi_{A} = \begin{bmatrix}
	0 & 1\\-1 & 0
	\end{bmatrix}_{2\times 2} \in M_{2}(\mathbb{C})$ and 
	\begin{equation*}
	\Big\langle \begin{bmatrix}
	0 & 1\\-1 & 0
	\end{bmatrix}\begin{bmatrix}
	\frac{1}{\sqrt{2}}\\
	\frac{1}{\sqrt{2}}
	\end{bmatrix},
	\begin{bmatrix}
	\frac{1}{\sqrt{2}}\\
	\frac{1}{\sqrt{2}}
	\end{bmatrix} \Big\rangle_{\mathbb{C}} = \Big\langle \begin{bmatrix}
	\frac{1}{\sqrt{2}}\\
	-\frac{1}{\sqrt{2}}
	\end{bmatrix}, \begin{bmatrix}
	\frac{1}{\sqrt{2}}\\
	\frac{1}{\sqrt{2}}
	\end{bmatrix}\Big\rangle_{\mathbb{C}} = \frac{1}{2} - \frac{1}{2} =0,
	\end{equation*}
	i.e.,  $0 \in W_{\mathbb{C}}(\rchi_{A})\subseteq  \Omega_{W_{\mathbb{C}}(\rchi_{A})}$. Now we show that $0 \notin W_{\mathbb{H}}(A)$. Assume that $0 \in W_{\mathbb{H}}(A)$, then $\overline{q}jq=0$, for some $q\in S_{\mathbb{H}}$. It follows that $0= |\overline{q}jq|= |q|^{2}$. This is contradiction to the fact that $|q|=1$. Therefore $W_{\mathbb{H}}(A)$ is a proper subset of  $\Omega_{W_{\mathbb{C}}(\rchi_{A})}$.
\end{eg}
Later we show that the equality in Equation (\ref{Equation:propersubset}) holds if and only if $W_{\mathbb{H}}(A)$ is convex (see Theorem \ref{Theorem: Main1}).
\begin{thm}\label{Theorem:Main}
	Let $A \in M_{n}(\mathbb{H})$. Then $W_{\mathbb{H}}(A)$ is compact in $\mathbb{H}$ and the numerical radius attains. Furthermore, the following properties hold true:
	\begin{enumerate}[(1)]
		\item $W_{\mathbb{H}}(\alpha I + \beta A) = \alpha + \beta W_{\mathbb{H}}(A)$, for every $\alpha, \beta \in \mathbb{R}$.
		\item If $B\in M_{n}(\mathbb{H})$, then $W_{\mathbb{H}}(A+B) \subseteq W_{\mathbb{H}}(A)+W_{\mathbb{H}}(B)$.
		\item $W_{\mathbb{H}}(U^{\ast}AU) = W_{\mathbb{H}}(A)$, for every unitary $U \in M_{n}(\mathbb{H})$.
		\item $W_{\mathbb{H}}(A^{\ast}) = W_{\mathbb{H}}(A)$.
	\end{enumerate}
\end{thm}
\begin{proof} Define $f_{A} \colon S_{\mathbb{H}^{n}}\to \mathbb{H} $ by
	\begin{equation*}
	f_{A}(X) = \langle X, AX\rangle_{\mathbb{H}}, \text{for all}\; X \in S_{\mathbb{H}^{n}}.
	\end{equation*}
	If $\{X_{m}\} \to X_{0}$ in $S_{\mathbb{H}^{n}}$, then $f_{A}(\{X_{m}\}) = \{\langle X_{m}, AX_{m}\rangle_{\mathbb{H}}\} \to \langle X_{0}, AX_{0}\rangle_{\mathbb{H}} = f_{A}(X_{0}) $. This implies that $f_{A}$ is continuous. Since $S_{\mathbb{H}^{n}}$ is compact, then $W_{\mathbb{H}}(A)$ is compact being the range of a continuous function $f_{A}$. By the definition of numerical radius of $A$, 
	\begin{equation*}
	w_{\mathbb{H}}(A) = \max\limits_{X \in S_{\mathbb{H}^{n}}} |\langle X, AX \rangle_{\mathbb{H}}| = \|\;|f_{A}|\;\|_{\infty}.
	\end{equation*}
	By the generalization of extreme value theorem, $|f_{A}|$ attains its maximum, so  ${\mathop{w}}_{\mathbb{H}}(A)$ is attained. 
	
	\noindent Proof of $(1):$ Since  $\langle X, (\alpha I + \beta A)X\rangle_{\mathbb{H}} = \alpha + \beta \langle X, AX\rangle_{\mathbb{H}} $ for every $X \in S_{\mathbb{H}^{n}}$, then 
	\begin{equation*}
	W_{\mathbb{H}}(\alpha I + \beta A) = \alpha + \beta W_{\mathbb{H}}(A).
	\end{equation*} 
	
	\noindent Proof of $(2):$ If  $q \in W_{\mathbb{H}}(A+B)$, then 
	\begin{equation*}
	q = \langle X, (A+B)X\rangle_{\mathbb{H}} = \langle X, AX\rangle_{\mathbb{H}}+ \langle X, BX\rangle_{\mathbb{H}}
	\end{equation*} for some $X \in S_{\mathbb{H}^{n}}$. Thus  $q \in W_{\mathbb{H}}(A)+W_{\mathbb{H}}(B)$. 
	
	\noindent Proof of $(3):$ For every $X \in S_{\mathbb{H}^{n}}$, there exists a unique $Y\in S_{\mathbb{H}^{n}}$ such that $UY = X$ and 
	\begin{equation*}
	\langle X, AX\rangle_{\mathbb{H}}= \langle UY, AUY \rangle_{\mathbb{H}} = \langle Y, U^{\ast}AUY \rangle_{\mathbb{H}}.
	\end{equation*}
	This shows that $W_{\mathbb{H}}(U^{\ast}AU)= W_{\mathbb{H}}(A)$.
	
	\noindent Proof of $(4):$ Let $q \in \mathbb{H}$, then 
	\begin{align*}
	q \in W_{\mathbb{H}}(A) &\iff q =\langle X, AX\rangle, \text {for some X}\; \in S_{\mathbb{H}^{n}}\\
	&\iff \overline{q} = {\langle X, A^{\ast}X\rangle} \\
	&\iff \overline{q} \in W_{\mathbb{H}}(A^{\ast}).
	\end{align*}
	This implies that $[q] \in W_{\mathbb{H}}(A)$ if and only if $[q] \in W_{\mathbb{H}}(A^{\ast})$ since $W_{\mathbb{H}}(A)$ is circular. Therefore, $W_{\mathbb{H}}(A^{\ast}) = W_{\mathbb{H}}(A)$.
\end{proof}
Now we provide some additional equivalent conditions for the convexity of quaternionic numerical range. 
\begin{thm}\label{Theorem: Main1}
	Let $A \in M_{n}(\mathbb{H})$. Then the following are equvivalent:
	\begin{enumerate}[(1)]
		\item $W_{\mathbb{H}}(A)$ is convex.
		\item $W_{\mathbb{H}}(A:\mathbb{C}) = W_{\mathbb{H}}(A)\cap \mathbb{C}$.
		\item $W_{\mathbb{H}}(A) = \Omega_{W_{\mathbb{C}}(\rchi_{A})}$.
		\item For every $X \in S_{\mathbb{H}^{n}}$, there exists a $Y\in S_{\mathbb{H}^{n}}$ such that 
		\begin{equation*}
		2\; |\text{im} (\langle Y, AY\rangle_{\mathbb{H}})| = |\langle X, AX\rangle_{\mathbb{H}}\; i - i\; \overline{\langle X, AX\rangle}_{\mathbb{H}}|. 
		\end{equation*}
	\end{enumerate}
\end{thm}
\begin{proof} It is proved in \cite[Theorem 2]{Au-Yeung1} that the quaternionic numerical range $W_{\mathbb{H}}(A)$ is convex if and only if $W_{\mathbb{H}}(A \colon \mathbb{C}) = W_{\mathbb{H}}(A)\cap \mathbb{C}$. That is $(1) \Leftrightarrow (2) $. Now we prove the rest of the equivalent conditions.

	\noindent $(2) \Rightarrow (3):$ Suppose that $W_{\mathbb{H}}(A:\mathbb{C}) = W_{\mathbb{H}}(A)\cap \mathbb{C}$. By Remark \ref{Remark: Onlysubset}, we have  $W_{\mathbb{H}}(A) \subseteq \Omega_{W_{\mathbb{C}}(A)}$. In general the reverse inclusion is not true (see Ecample \ref{Example: NeedNot} ), but by assuming $(2)$ we prove that it holds true.  If  $q \in \Omega_{W_{\mathbb{C}}(\rchi_A)}$, then $z_{\pm}: = \text{re}(q) \pm i \;|\text{im}(q)|  \sim q $ and $z_{\pm} \in W_{\mathbb{C}}(\rchi_{A})$. By Propositon \ref{Proposition:projection}, we have   
	\begin{equation*}
	z_{\pm} \in W_{\mathbb{H}}(A:\mathbb{C})=W_{\mathbb{H}}(A)\cap \mathbb{C}.
	\end{equation*} 
	Since $z_{\pm} \sim q$ and $W_{\mathbb{H}}(A)$ is circular, we have  $q \in W_{\mathbb{H}}(A)$. 
	
	\noindent $(3) \Rightarrow (4):$  Let $X \in S_{\mathbb{H}^{n}}$. Then by Proposition \ref{Proposition:projection},  $co(\langle X, AX\rangle_{\mathbb{H}}) \in W_{\mathbb{C}}(\rchi_{A})$. Since $W_{\mathbb{H}}(A) = \Omega_{W_{\mathbb{C}}(\rchi_{A})}$, there exist a $Y \in S_{\mathbb{H}^{n}}$ such that $\langle Y, AY\rangle_{\mathbb{H}} \sim co(\langle X, AX\rangle_{\mathbb{H}})$. This implies that 
	\begin{equation*}
	\text{re}(\langle X, AX\rangle_{\mathbb{H}}) = \text{re}(\langle Y, AY\rangle_{\mathbb{H}})
	\end{equation*}
	and 
	\begin{equation} \label{Equation: im1}
	|\text{im} \big(co (\langle X, AX \rangle_{\mathbb{H}})\big)| =   |\text{im} (\langle Y, AY\rangle_{\mathbb{H}})|.
	\end{equation}
	By $(2)$ of Note 1.1, we write 
	\begin{equation} \label{Equation: im2}
	| \text{im} \big(co (\langle X, AX \rangle_{\mathbb{H}})\big)| = \frac{1}{2} |\langle X, AX\rangle_{\mathbb{H}} i - i \overline{\langle X, AX\rangle}_{\mathbb{H}}|.
	\end{equation}
	From Equations (\ref{Equation: im1}), (\ref{Equation: im2}), we have
	\begin{equation*}
	2\; |\text{im} (\langle Y, AY\rangle_{\mathbb{H}})| = |\langle X, AX\rangle_{\mathbb{H}}\; i - i\; \overline{\langle X, AX\rangle}_{\mathbb{H}}|. 
	\end{equation*}
	\noindent $(4) \Rightarrow (2):$ By Definition \ref{Definition:complexpart}, it is clear that  $W_{\mathbb{H}}(A)\cap \mathbb{C} \subseteq W_{\mathbb{H}}(A\colon \mathbb{C})$. To show the reverse inclusion, let us take $ z \in W_{\mathbb{H}}(A\colon \mathbb{C})$, then  $z = co (\langle X, AX\rangle_{\mathbb{H}})$  for some $X\in S_{\mathbb{H}^{n}}$. By the assumption of $(4)$, there exists a $Y\in S_{\mathbb{H}^{n}}$  such that 
	\begin{equation*}
	|\text{im}(\langle Y, AY\rangle_{\mathbb{H}})|= \frac{1}{2}|\langle X, AX\rangle_{\mathbb{H}}\; i - i\; \overline{\langle X, AX\rangle}_{\mathbb{H}}|= |\text{im}(z)|.
	\end{equation*}
	From the proof of $(3)\Rightarrow (4)$, the existence of $Y$ for a given $X$ guarantees that $\text{re}(\langle Y, AY\rangle_{\mathbb{H}} ) = \text{re}(\langle X, AX\rangle_{\mathbb{H}})$. It implies that $z \sim \langle Y, AY\rangle_{\mathbb{H}}$. Since $W_{\mathbb{H}}(A)$ is circular, $z \in W_{\mathbb{H}}(A)\cap \mathbb{C}$. Therefore $W_{\mathbb{H}}(A:\mathbb{C}) = W_{\mathbb{H}}(A)\cap \mathbb{C}$.
\end{proof}
\begin{eg}\label{Example:main}
	If $A = \begin{bmatrix}
	k & 0\\0 & k
	\end{bmatrix}_{2\times 2} $ , then
	$
	W_{\mathbb{H}}(A) = \big\{ q \in \mathbb{H}:\; \overline{q} = -q ,\; |q|\leq 1 \big\}
	$
	is convex (see Case $(1)$ of Lemma \ref{Lemma:keylemma} for details).
\end{eg}
Now we show that the intersection of complex section of quaternionic numerical range and any line parallel to $y$ - axix is either emptyset or connected.  We recall a lemma which is useful in this context. 
\begin{lemma}\cite[Lemma 3]{Au-Yeung2} \label{Lemma: partlyconnected}
	Let $A \in M_{n}(\mathbb{H})$. Then the following assertions hold true:
	\begin{enumerate}
		\item If  $ A^{\ast} = A$, then the set $\{X\in S_{\mathbb{H}^{n}}: \langle X, AX\rangle_{\mathbb{H}} = \alpha\}$ is connected for any $\alpha \in \mathbb{R}$.
		\item If $A^{*}= -A$, then the set $\{X\in S_{\mathbb{H}^{n}}: \langle X, AX\rangle_{\mathbb{H}} = 0\}$ is connected.  
	\end{enumerate}
\end{lemma}
\begin{cor}\label{Corollary: partlyconnected}
	Let $A\in M_{n}(\mathbb{H})$. Then $W_{\mathbb{H}}(A)\cap \mathbb{R}$ is either {\it emptyset} or {\it connected}. 
\end{cor}
\begin{proof}
	Given $A \in M_{n}(\mathbb{H})$ is decomposed as,
	\begin{equation*}
	A = \frac{1}{2}(A + A^{\ast} )+ \frac{1}{2}(A-A^{\ast}).
	\end{equation*}
	Here $(A+A^{\ast})$ is a self-adjoint and $(A-A^{\ast})$ is an anti self-adjoint matrix.
	Since $W_{\mathbb{H}}(A) \cap \mathbb{R}: = \{X\in S_{\mathbb{H}^{n}}:\; \langle X, (A-A^{\ast})X\rangle_{\mathbb{H}} = 0\}$, then by   $(2)$ of Lemma \ref{Lemma: partlyconnected}, $W_{\mathbb{H}}(A)\cap \mathbb{R}$ is connected.
\end{proof}
\begin{lemma} \label{Lemma: connected}
	Let $A \in M_{n}(\mathbb{H})$ and $L$ be the line parallel to $y$ - axis. Then $W_{\mathbb{H}}^{+}(A) \cap L $ is either {\it emptyset} or {\it connected}.   
\end{lemma}
\begin{proof}
	let $\alpha \in \mathbb{R}$ be fixed. Then  $L_{\alpha} := \{\alpha + i \beta \;;\; \beta \in \mathbb{R}\} $ is a line  parallel to $y$ - axis. We know that  $A = \frac{1}{2}(A+A^{\ast}) + \frac{1}{2}(A-A^{\ast})$ and let
	\begin{equation*}
	U := \{X\in S_{\mathbb{H}^{n}}: \langle X, (A+A^{\ast})X\rangle_{\mathbb{H}}= 2\alpha\}.
	\end{equation*}
	Then by $(1)$ of Lemma \ref{Lemma: partlyconnected}, $U$ is connected. Define $\Phi \colon \mathbb{H}^{n} \to \mathbb{C}^{+}$ by 
	\begin{equation*}
	\Phi(Y) = \text{re}(\langle Y, AY\rangle_{\mathbb{H}}) + i \; |\text{im}(\langle Y, AY\rangle_{\mathbb{H}})|.
	\end{equation*}
	Clearly, $\Phi$ is well-defined and continuous since 
	\begin{equation*}
	|\Phi(Y)| = |\langle Y, AY\rangle_{\mathbb{H}}| \leq \|A\|\|Y\|^{2}, \; \text{for every}\; Y \in \mathbb{H}^{n}.
	\end{equation*}
	Then $\Phi(U)$ is connected. 
	Suppose that $W_{\mathbb{H}}^{+}(A) \cap L_{\alpha}$ is non-empty, then we claim that $\Phi(U) = W_{\mathbb{H}}^{+}(A) \cap L_{\alpha}$. Let $\alpha + i \beta^{\prime}\in \mathbb{C}^{+}$. Then 
	\begin{align*}
	\alpha + i \beta^{\prime} &\in W^{+}_{\mathbb{H}}(A)\cap L_{\alpha} \\ &\iff \alpha + i \beta^{\prime} = \langle Y, AY \rangle_{\mathbb{H}} \in \mathbb{C}^{+}, \; \text{for some}\; Y\in S_{\mathbb{H}^{n}}\\
	&\iff \alpha = \text{re}(\langle Y, AY \rangle_{\mathbb{H}}), \; \beta^{\prime} = |\text{im}(\langle Y, AY\rangle_{\mathbb{H}})|\\
	&\iff 2\alpha = \langle Y, AY \rangle_{\mathbb{H}} + \overline{\langle Y, AY \rangle}_{\mathbb{H}},  \; \beta^{\prime} = |\text{im}(\langle Y, AY\rangle_{\mathbb{H}})|\\
	&\iff 2\alpha = \langle Y, (A+A^{\ast})Y \rangle_{\mathbb{H}}, \; \beta^{\prime} = |\text{im}(\langle Y, AY\rangle_{\mathbb{H}})|\\
	&\iff Y \in U, \; \Phi(Y) = \alpha + i \beta^{\prime}\\
	&\iff \alpha + i \beta^{\prime} \in \Phi(U).
	\end{align*}
	This shows that $W_{\mathbb{H}}^{+}(A)\cap L_{\alpha} = \Phi(U)$ which is connected. Since $\alpha \in \mathbb{R}$ is arbitrary, we conclude that $W_{\mathbb{H}}^{+}(A) \cap L$ is either empty set or  connected for any line $L$ parallel to $y$ - axis.
\end{proof}
Note that for every $m \in \mathbb{S}$, the intersection of $\mathbb{C}_{m}$- section of quaternionic numerical range and the vertical line is either emptyset or connected. This can be seen by using the identification of $\mathbb{C}_{m}$ with $\mathbb{C}$ and by Lemma \ref{Lemma: connected}.
\begin{prop}\label{Prop:conv}
	Let $S$ be (not necessarily convex) a finite subset of $\mathbb{C}$. Then
	\begin{equation*}
	conv(\Omega_{conv(S)}) = conv(\Omega_{S}).
	\end{equation*} 
	Here $conv(\cdot)$ is an abbreviation for `{\it convex hull of }' .
\end{prop}
\begin{proof} It is easy to see that $conv(\Omega_{S}) \subseteq conv(\Omega_{conv(S)})$ as follows: If $q \in \Omega_{S}$, then $q = \alpha + m \; \beta $ for some $m \in \mathbb{S}$ and $\alpha + i \beta \in S \subseteq conv(S)$. Thus   $q \in \Omega_{conv(S)} \subseteq conv(\Omega_{conv(S)})$. Since $conv(\Omega_{S})$ is the smallest convex set containing $\Omega_{S}$, we have 
	\begin{equation*}
	conv(\Omega_{S}) \subseteq conv(\Omega_{conv(S)}).
	\end{equation*}   
	Now we prove the reverse inequality. Suppose $S = \{a_{1}+ib_{1}, a_{2}+ i b_{2}, \cdots,  a_{r}+ i b_{r}\}$ and let $p \in \Omega_{conv(S)}$. Then  $p = a + n\; b$ for some $n \in \mathbb{S}$ and $a+ib \in conv(S)$, where 
	\begin{equation*}
	a+ib = \sum\limits_{\ell=1}^{r} \lambda_{\ell }\; (a_{\ell} + i b_{\ell}), \;\;\; \text{where}\; \sum\limits_{\ell=1}^{r}\lambda_{\ell} = 1, \; \lambda_{\ell} \geq 0.
	\end{equation*} 
	Take $p_{\ell}: = a_{\ell } + n b_{\ell}$, then $p_{\ell} \in \Omega_{S}$, for each $\ell \in \{1,2,3,\cdots, r\}$ and
	\begin{align*}
	\sum\limits_{\ell=1}^{r} \lambda_{\ell}\;  p_{\ell} = \sum\limits_{\ell=1}^{r}\lambda_{\ell} a_{\ell} + n \sum\limits_{\ell =1}^{r} \lambda_{\ell} b_{\ell} 
	= a + n b = p.
	\end{align*}
	Therefore $p \in conv(\Omega_{S})$. This implies that $\Omega_{conv(S)} \subseteq conv(\Omega_{S})$. Hence  $conv(\Omega_{conv(S)}) \subseteq conv(\Omega_{S})$.
\end{proof}
\begin{thm}
	Let $A \in M_{n}(\mathbb{H})$ be normal. Then 
	\begin{equation*}
	conv\big(\Omega_{W_{\mathbb{H}}(A : \; \mathbb{C})}\big) = conv(\sigma_{S}(A)).
	\end{equation*}

	\begin{proof}
		Since $A$ is normal, then $\rchi_{A}$ is normal and by \cite[Theorem 1.4-4]{GustafsonRao}, we have $W_{\mathbb{C}}(\rchi_{A}) = conv(\sigma(\rchi_{A}))$. From Proposition \ref{Proposition:projection} and \ref{Prop:conv}, we see that  
		\begin{align*}
		conv(\Omega_{W_{\mathbb{H}}(A: \mathbb{C})}) &= conv(\Omega_{W_{\mathbb{C}}(\rchi_{A})}) \\
		&= conv(\Omega_{conv(\sigma(\rchi_{A}))})\\
		&= conv(\Omega_{\sigma(\rchi_{A})})\\
		&= conv(\sigma_{S}(A)) \;  \; (\text{by (3) of Note \ref{Note: KeyNote} }).\qedhere
		\end{align*}
		
	\end{proof}
\end{thm}
\begin{lemma}\label{Lemma:keylemma} Let $A \in M_{2}(\mathbb{H})$. Then every section of $W_{\mathbb{H}}(A)$ is convex.   
\end{lemma}
\begin{proof}
	By Schur canonical form \cite[Theorem 6.1]{Zhang} for matrices over quaternions, there exist a unitary matrix $U \in M_{n}(\mathbb{H})$ such that 
	\begin{equation*}
	A = U^{\ast} \begin{bmatrix}
	{z}_{1} & p \\
	0 & z_{2} 
	\end{bmatrix}U, 
	\end{equation*}
	for some $p \in \mathbb{H}$ and $z_{1}, z_{2} \in \mathbb{C}^{+}$ are the standard eigenvalues of $A$. Since every slice $\mathbb{C}_{m}$ for $m \in \mathbb{S}$ is identified with the complex field $\mathbb{C}$, it is enough to show $W^{+}_{\mathbb{H}}(A)= W_{\mathbb{H}}(A) \cap \mathbb{C}^{+}$ is a convex subset of $\mathbb{C}$.  Further by $(3)$ of Theorem \ref{Theorem:Main}, it is sufficient to show that the complex section of quaternionic numerical range of $ \begin{bmatrix}
	z_{1} & p \\
	0 & z_{2} 
	\end{bmatrix}$ is convex.

	Let $\begin{bmatrix}
	x\\ y
	\end{bmatrix} \in S_{\mathbb{H}^{2}}$. Then our proof is divided into the following cases.  
	
	\noindent Case $(1):$ If $z_{1} = z_{2}= z:= a+i b \in \mathbb{C}^{+}$ and $p = 0$, then 
	\begin{align*}
	\big\langle \begin{bmatrix}
	x \\ y 
	\end{bmatrix},  \begin{bmatrix}
	a+ib & 0\\
	0 & a+ib
	\end{bmatrix}\begin{bmatrix}
	x \\  y 
	\end{bmatrix}  \big\rangle_{\mathbb{H}}  =  a (|x|^{2}+|y|^{2})+ b\; m_{x,y},
	\end{align*}
	where $m_{x,y} : = \overline{x}ix + \overline{y}iy$.
	Clearly,  $\overline{m}_{x,y} =-m_{x,y}$ and $|m_{x,y}| \leq |x|^{2}+|y^{2}|=1$. It implies that $\big\{m_{x,y} :\; |x|^{2}+|y|^{2}=1\big\} \subseteq \{q \in \mathbb{H}: \; \overline{q} = -q,\; |q|\leq 1\}$. Let $q \in \mathbb{H}\setminus \{0\}$ be such that  $\overline{q} = - q$ and $|q| \leq 1$. Then there exist a $s_{q}\in \mathbb{H}\setminus \{0\}$ such that ${s}^{-1}_{q}\; i\; s_{q} = \frac{q}{|q|}$. Take
	\begin{equation*}
	x = \sqrt{\frac{1+|q|}{2}}\cdot \frac{s_{q}}{|s_{q}|} \; ; \; y = j\sqrt{\frac{1-|q|}{2}}\cdot \frac{s_{q}}{|s_{q}|}
	\end{equation*} 
	then $|x|^{2}+|y|^{2} =1$ and 
	\begin{align*}
	m_{x,y} = \overline{x}ix + \overline{y}iy 
	&= \Big(\frac{1+|q|}{2}\Big){s}^{-1}_{q}\; i\;s_{q}- \Big(\frac{1-|q|}{2}\Big){s}^{-1}_{q}\; i\;s_{q}\\
	&= |q|\;  \frac{q}{|q|}\\
	&= q.
	\end{align*}
	If $q = 0$, then by taking $x = \frac{1}{\sqrt{2}},\; y = j \frac{1}{\sqrt{2}}$, we see that $|x|^{2}+ |y|^{2} = 1$ and $m_{x,y} = 0$.
	This proves the reverse inclusion.  As a result we have  
	\begin{equation*}
	\big\{m_{x,y} :\; |x|^{2}+|y|^{2}=1\big\}=\big\{q \in \mathbb{H}: \; \overline{q} = -q,\; |q|\leq 1\big\}.
	\end{equation*}
	Therefore, 
	\begin{align*}
	W_{\mathbb{H}}(A)&= \{a+bq:\; \overline{q} =- q, \; 0 \leq |q|\leq 1\}.
	\end{align*} 
	It is the solid sphere in $\mathbb{R}^{4}$ with radius $b$ and center  at $(a, 0, 0, 0)$.  So $W_{\mathbb{H}}(A)$ is convex. In particular, 
	\begin{equation*}
	W^{+}_{\mathbb{H}}(A) = \big\{a + i b\beta  : 0\leq \beta\leq 1\big\}.
	\end{equation*} 
	It is the line segment joining $\text{re}(z)$ and $z$.
	
	\noindent Case $(2):$ Let $z_{1}= a_{1}+i b_{1}, z_{2}= a_{2} + i b_{2} \in \mathbb{C}^{+}$. If $z_{1} \neq z_{2}$ and $p = 0$, then  
	\begin{equation}\label{Equation: Case2}
	\big\langle \begin{bmatrix}
	x\\ y
	\end{bmatrix}, \begin{bmatrix}
	a_{1}+ib_{1} & 0\\
	0 & a_{2}+i b_{2} 
	\end{bmatrix}\begin{bmatrix}
	x\\ y
	\end{bmatrix}\big\rangle_{\mathbb{H}} 
	= |x|^{2}a_{1}+ b_{1} \overline{x}ix+ |y|^{2}a_{2} +   b_{2} \overline{y}i y.
	\end{equation}
	Suppose that the imaginary part of Equation (\ref{Equation: Case2}) is zero i.e., 
	\begin{equation}\label{Equation: imaginaryZero}
	b_{1}\overline{x}ix= - b_{2}\overline{y}iy.
	\end{equation}
	By taking modulus on both sides of Equation (\ref{Equation: imaginaryZero}), we get  $b_{1}|x|^{2}= b_{2}|y|^{2}$. Since  $|x|^{2}+|y|^{2}=1$, we have 
	\begin{equation}\label{Equation:|x||y|}
	|x| = \sqrt\frac{b_{2}}{b_{1}+ b_{2}}\; \; ; \;\; |y| = \sqrt\frac{b_{1}}{b_{1}+b_{2}}.
	\end{equation}
	Moreover, from Equation (\ref{Equation: imaginaryZero}), (\ref{Equation:|x||y|}) we have
	\begin{equation*}
	\frac{b_{1}b_{2}}{b_{1}+b_{2}}\; yx^{-1}\; i = -\frac{b_{1}b_{2}}{b_{1}+b_{2}}\; iyx^{-1}.
	\end{equation*}  
	That is, $x^{-1}\; ix+ y^{-1}iy = 0$. Conversely, if we choose $x,y$ as in Equation (\ref{Equation:|x||y|}) and satisfying  $x^{-1}\; ix+ y^{-1}iy = 0$, then $ b_{1}\overline{x}ix + b_{2} \overline{y}iy = 0$. Thus 
	\begin{equation*}
	W_{\mathbb{H}}(A) \cap \mathbb{R} = \Big\{v:=\frac{a_{1} b_{2}+ b_{1}a_{2}}{b_{1}+b_{2}} \Big\}.
	\end{equation*}  
	We show that $W^{+}_{\mathbb{H}}(A)= conv \big\{z_{1}, z_{2},  v \big\}$. For instance, if we choose $x,y \in \mathbb{C}$ with $|x|^{2}+|y|^{2}=1$, then by Equation (\ref{Equation: Case2}), we see that
	\begin{equation*}
	(a_{1}+i b_{1})|x|^{2} + (a_{2}+ i b_{2})|y|^{2} \in W_{\mathbb{H}}^{+}(A) 
	\end{equation*}
	i.e., the  line segment joining $z_{1}, z_{2} $ is in $W_{\mathbb{H}}^{+}(A)$. Now we prove that the line segments joining $v$ with $a_{1}+i b_{1}$ and $a_{2} + i b_{2}$ are in $W_{\mathbb{H}}^{+}(A)$ as follows: 
	
	Let $u_{t} := a_{1} (1-t) + vt$, $\forall \;t \in [0,1]$ and 
	\begin{equation*}
	x_{t}= \sqrt{\frac{a_{2}-u_{t}}{a_{2}-a_{1}}}\; , \; \; y_{t}= j \;\sqrt{\frac{u_{t}-a_{1}}{a_{2}-a_{1}}}.
	\end{equation*}
	Then $|x_{t}|^{2}+|y_{t}|^{2} = 1$. Moreover, 
	\begin{align*}
	\big\langle \begin{bmatrix}
	x_{t} \\ y_{t}
	\end{bmatrix},   &\begin{bmatrix}
	a_{1}+ i b_{1}& 0\\
	0 & a_{2} + i b_{2}
	\end{bmatrix} \begin{bmatrix}
	x_{t} \\ y_{t}
	\end{bmatrix}\big\rangle_{\mathbb{H}}\\&=|x_{t}|^{2}a_{1}+ |y_{t}|^{2}a_{2}+ b_{1}\overline{x_{t}}ix_{t} + b_{2}\overline{y_{t}}iy_{t}\; \\
	&= u_{t} +\frac{(a_{1}b_{2}+a_{2}b_{1})}{a_{2}-a_{1}} i - \frac{(b_{1}+b_{2})}{a_{2}-a_{1}}iu_{t} \\
	&=u_{t} +\frac{(a_{1}b_{2}+a_{2}b_{1})}{a_{2}-a_{1}} i -\frac{(b_{1}+b_{2})}{a_{2}-a_{1}}i vt -\frac{(b_{1}+b_{2})}{a_{2}-a_{1}}i a_{1}(1-t)\\
	&= u_{t} + \Big[ \frac{a_{1}b_{2}+a_{2}b_{1}-a_{1}b_{1}-a_{1}b_{2}}{a_{2}-a_{1}}\Big] i(1-t) \\
	&= u_{t} + b_{1}i(1-t) \\
	&= (a_{1}+ib_{1})(1-t)+vt.
	\end{align*}
	Similarly, we see that the segment joining $a_{2}+i b_{2}$ and $v$ is in $W_{\mathbb{H}}^{+}(A)$. It is clear from Lemma \ref{Lemma: connected} that  $conv\{z_{1}, z_{2}, v\} \subseteq  W_{\mathbb{H}}^{+}(A)$. Suppose that $z \in \mathbb{C}^{+}$ and $z \notin conv\{z_{1},z_{2},v\}$, then $z\notin conv(\Omega_{conv\{z_{1},z_{2},v\}}) = conv(\Omega_{\{z_{1}, z_{2}, v\}}) \supseteq W^{+}_{\mathbb{H}}(A) $. Thus $z \notin W_{\mathbb{H}}^{+}(A)$. It shows that $conv\{z_{1}, z_{2}, v\} =  W_{\mathbb{H}}^{+}(A)$.
	
	\noindent Case $(3):$ If $z_{1} = z_{2} = 0$, then by the Young's inequality, we have  
	\begin{equation*}
	\big|\big\langle \begin{bmatrix}
	x\\ y
	\end{bmatrix}, \begin{bmatrix}
	py\\ 0
	\end{bmatrix}\big\rangle_{\mathbb{H}}\big| = |\overline{x}py|\leq |p||x||y|\leq |p|\big(\frac{|x|^{2}+|y|^{2}}{2}\big) = \frac{|p|}{2}. 
	\end{equation*}
	Thus 
	\begin{equation*}
	W_{\mathbb{H}}(A)\subseteq \Big\{q\in \mathbb{H}: |q|\leq \frac{|p|}{2}\Big\}. 
	\end{equation*}
	Suppose that $|p|= 1 $ and $q \in \mathbb{H}$ with $|q| \leq \frac{1}{2}$, then $q = r e^{\mathfrak{m}_{q}\theta },\; 0 \leq r \leq \frac{1}{2}$ where $\mathfrak{m}_{q} = \frac{\text{im}(q)}{|\text{im}(q)|}$. If $x = e^{-\mathfrak{m}_{q}\theta}{\text cos}\;\alpha $ and $y = p^{-1}{\text sin}\; \alpha$ such that ${\text sin}\;2\alpha = 2r \leq 1$ and $0 \leq \alpha \leq \frac{\pi}{4}$, then $|x|^{2}+|y|^{2}= {\text cos}^{2}\alpha + {\text sin}^{2}\alpha  = 1$ and 
	\begin{equation*}
	\overline{x}p y = e^{\mathfrak{m}_{q}\theta} {\text sin } \alpha\;  {\text cos } \alpha = re^{\mathfrak{m}_{q}\theta} = q . 
	\end{equation*} 
	It shows that $W_{\mathbb{H}}(A)= \Big\{q\in \mathbb{H}: |q|\leq \frac{1}{2}\Big\}$. If $|p| \neq 1$, then by Theorem \ref{Theorem:Main}, we have
	\begin{equation*}
	W_{\mathbb{H}}(A) = W_{\mathbb{H}}\Big(\begin{bmatrix}
	0 & \frac{p}{|p|}\\
	0 &0
	\end{bmatrix}\Big) |p|= \Big\{q\in \mathbb{H}: |q|\leq \frac{|p|}{2}\Big\}.
	\end{equation*}  
	Thus $
	W^{+}_{\mathbb{H}}(A) = \Big\{z\in \mathbb{C}^{+}: |z|\leq \frac{|p|}{2}\Big\}$.
	It is the upper half of the disk of radius $\frac{|p|}{2}$ in the complex plane.  
	
	\noindent Case $(4):$ Let $z_{1}= a_{1}+i b_{1}, z_{2}= a_{2} + i b_{2} \in \mathbb{C}^{+}$. If $z_{1} \neq z_{2}$ and $p \neq 0$, then 
	\begin{align*}\label{Equation:case4}
	\big\langle \begin{bmatrix}
	x\\y
	\end{bmatrix}, &\begin{bmatrix}
	a_{1}+ib_{1} & p \\
	0 & a_{2}+ib_{2}
	\end{bmatrix}\begin{bmatrix}
	x\\y
	\end{bmatrix}\big\rangle_{\mathbb{H}} \\
	&= a_{1}|x|^{2}+ a_{2}|y|^{2}+b_{1}\overline{x}ix+ \overline{x}py+ b_{2}\overline{y}iy.
	\end{align*}
	The imaginary part of the above innerproduct is given by 
	\begin{align*}
	im \Big(\big\langle \begin{bmatrix}
	x\\y
	\end{bmatrix}, \begin{bmatrix}
	a_{1}+ib_{1} & p \\
	0 & a_{2}+ib_{2}
	\end{bmatrix}&\begin{bmatrix}
	x\\y
	\end{bmatrix}\big\rangle_{\mathbb{H}}\Big) \\
	&= b_{1}\overline{x}ix+ b_{2}\overline{y}iy+ \frac{1}{2}\big( \overline{x}py-\overline{y}\;\overline{p}x\big)\\
	&= \big\langle \begin{bmatrix}
	x\\y
	\end{bmatrix}, \begin{bmatrix}
	ib_{1}& \frac{p}{2}\\
	\frac{-\overline{p}}{2}&ib_{2}
	\end{bmatrix}\begin{bmatrix}
	x\\y
	\end{bmatrix}\big\rangle_{\mathbb{H}}.
	\end{align*}
	Let $B:= \begin{bmatrix}
	ib_{1}& \frac{p}{2}\\
	\frac{-\overline{p}}{2}&ib_{2}
	\end{bmatrix} $. Then $B$ is anti self-adjoint and 
	\begin{equation*}
	W_{\mathbb{H}}(A) \cap \mathbb{R} = \Big\{\big\langle \begin{bmatrix}
	x\\y
	\end{bmatrix}, A\begin{bmatrix}
	x\\y
	\end{bmatrix}\big\rangle_{\mathbb{H}}:\; \;\begin{bmatrix}
	x\\y
	\end{bmatrix}\in S_{\mathbb{H}^{2}}\; \&\; \big\langle \begin{bmatrix}
	x\\y
	\end{bmatrix}, B\begin{bmatrix}
	x\\y
	\end{bmatrix}\big\rangle_{\mathbb{H}} = 0\Big\}.
	\end{equation*}
	If $p = p_{0}+p_{1}i+p_{2}j+p_{3}k$, then the complex matrix $\rchi_{B}$ associated to $B$ is given by  
	\begin{equation*}
	\rchi_{B} = \begin{bmatrix}
	ib_{1} & \frac{1}{2}(p_{0}+p_{1}i)& 0 & \frac{1}{2}(p_{2}+p_{3}i)\\
	\frac{1}{2}(-p_{0}+p_{1}i) & i b_{2}& \frac{1}{2}(p_{2}+p_{3}i)&0\\
	0 & \frac{1}{2}(- p_{2}+p_{3}i)&- ib_{1}&\frac{1}{2}(p_{0}-p_{1}i)\\
	\frac{1}{2}(- p_{2}+p_{3}i) & 0 & \frac{1}{2}(-p_{0}-p_{1}i)& - ib_{2}
	\end{bmatrix}_{4 \times 4.}
	\end{equation*} 
	The determinant of $\rchi_{B}$ is computed as,
	\begin{equation*}
	det(\rchi_{B})= (b_{1} b_{2})^{2}+ \frac{1}{2}b_{1} b_{2} |p|^{2}+ \frac{1}{16}|p|^{4} > 0,
	\end{equation*} 
	since $|p|\neq 0$. It implies that $\rchi_{B}$ is invertible.  Equivalently, $B$ is invertible. Moreover, by Corollary \ref{Corollary: partlyconnected}, we have $W_{\mathbb{H}}(A) \cap \mathbb{R}$ is an interval in $\mathbb{R}$. 
	
	Now we claim that $W_{\mathbb{H}}^{+}(A)$ is convex:
	
	Let $\Gamma : = \Big\{\mathfrak{u} + \mathfrak{r}: \mathfrak{u}\in W_{\mathbb{H}}^{+}(\begin{bmatrix}
	z_{1}& 0\\
	0 & z_{2}
	\end{bmatrix}), \mathfrak{r}\in \mathbb{C}, \; |\mathfrak{r}| \leq \frac{|p|}{2}\Big\}$. Then  $W_{\mathbb{H}}^{+}(A)$ is a closed subset of $ \Gamma$ .  So it is enough to show that  $\Gamma$ is convex. If $\mathfrak{u}_{1}+\mathfrak{r}_{1}, \;\mathfrak{u}_{2}+\mathfrak{r}_{2} \in \Gamma$, then by Case (3), we have  $\alpha \mathfrak{u}_{1} + (1-\alpha)\mathfrak{u}_{2}\in W_{\mathbb{H}}^{+}(\begin{bmatrix}
	z_{1}& 0\\
	0 & z_{2}
	\end{bmatrix})$ and $|\alpha \mathfrak{r}_{1}+ (1-\alpha)\mathfrak{r}_{2}|\leq \alpha \frac{|p|}{2} + (1-\alpha)\frac{|p|}{2}= \frac{|p|}{2}$ for $\alpha \in [0,1]$. It implies that $\alpha (\mathfrak{u}_{1}+\mathfrak{r}_{1})+ (1-\alpha) (\mathfrak{u}_{2}+\mathfrak{r}_{2}) \in \Gamma$.   Hence $W_{\mathbb{H}}^{+}(A)$ is convex.  
	\paragraph{\bf Figures} The complex section of the quaternionic numerical range $W_{\mathbb{H}}^{+}(A)$ in first three  cases of Lemma \ref{Lemma:keylemma} is drawn as follows:
	
	\begin{center}
		
		\begin{tikzpicture}{case}
		\draw[->, thick] (-1,0) -- (3.5,0);
		\draw[->, thick] (0,-1) -- (0,3);
		\draw[blue] (2,0)    -- (2,2);
		\draw[fill] (2,0) circle [radius=0.025];
		\draw[fill] (2,2) circle [radius=0.025];
		\node [below] at (2,0) {$\text{re}(z)$};
		\node [above] at (2,2) {$z$};
		\node at (1,-1.75) {Case (1). };
		\end{tikzpicture}
		\qquad 
		\hspace{1cm}
		\begin{tikzpicture}
		\coordinate (r0) at (1,0);
		\coordinate (s0) at (1.5,2);
		\coordinate (si) at (-1, 1.5);
		\filldraw[draw=blue, fill= blue!30] (r0) -- (s0) -- (si) -- (r0) -- cycle;
		\draw[->, thick] (-2,0) -- (3,0);
		\draw[->, thick] (0,-1) -- (0,3);
		\draw[fill] (1,0) circle [radius= 0.025];
		\draw[fill] (1.5,2) circle [radius=0.025];
		\draw[fill] (-1,1.5) circle [radius=0.025];
		\node[below] at (1,0) {\it{v}};
		\node[right] at (1.5, 2) {$z_{2}$};
		\node[left] at (-1, 1.5) {$z_{1}$};
		\node at (0.5,-1.75) {Case (2). };
		\end{tikzpicture}

		\begin{tikzpicture}
		\draw[->, thick] (-2.5,0) -- (3,0);
		\draw[->, thick] (0,-1) -- (0,3);
		\draw[solid, blue]  (2,0) arc (0: 180: 2);
		\draw[->, dashed] (0,0) -- node[right]{$\frac{|p|}{2}$} (1.25, 1.55);
		\path [fill=blue, fill opacity = 0.3] (2,0) arc (0: 180: 2);
		\node at (0,-1.5) {Case (3). }; 
		\end{tikzpicture} 
	\end{center} \qedhere
\end{proof}

In the classical theory, it is evident from the Toeplitz-Hausdroff theorem  that the numerical range of complex matrix is convex. In case of quaternionic matrices, we prove Toeplitz-Hausdorff like theorem that is, every section of quaternionic numerical range is convex.   
\begin{thm}(Toeplitz-Hausdorff like theorem) \label{Theorem: TH}
	Let $A \in M_{n}(\mathbb{H})$. Then every section of $A$ is convex.
\end{thm}
\begin{proof} As every section $W_{\mathbb{H}}(A) \cap \mathbb{C}^{+}_{m}$ for $m \in \mathbb{S}$ is isomorphic to $W_{\mathbb{H}}^{+}(A)$, it is sufficient to show $W^{+}_{\mathbb{H}}(A)$ is convex. Let $z_{1}, z_{2} \in W^{+}_{\mathbb{H}}(A)$
	\begin{equation*}
	z_{1}:=\big\langle X, AX \big\rangle_{\mathbb{H}} \; \text{and}\; z_{2}:=\big\langle Y, AY\big\rangle_{\mathbb{H}}
	\end{equation*}
	for some $X, Y\in S_{\mathbb{H}^{n}}$. We show that the line segment joining $z_{1}$ and $z_{2}$ is contained in $W^{+}_{\mathbb{H}}(A)$.  Let  $V$ be two dimensional right quaternionic Hilbert space generated by $\{X, Y\}$ and $P$ be the orthogonal projection of $\mathbb{H}^{n}$ onto $V$. Then $V$ is isomorphic to $\mathbb{H}^{2}$. Since $A$ is a right quaternionic linear operator on $\mathbb{H}^{n}$, we see that the restriction of $PAP$ onto $V$ is a right quaternionic linear operator on $V$ such that 
	\begin{align*}
	\langle X, PAP X\rangle  &= \langle PX, APX\rangle = \langle X, AX\rangle = z_{1} ,\\
	\langle Y, PAP Y\rangle  &= \langle PY, APY\rangle = \langle Y, AY\rangle = z_{2}. 
	\end{align*} 
	It implies that $z_{1}, z_{2} \in W_{\mathbb{H}}^{+}(PAP|_{V})$. Since $PAP|_{V} \in M_{2}(\mathbb{H})$, by Lemma \ref{Lemma:keylemma}, we see that $W^{+}_{\mathbb{H}}(PAP|_{V})$ is convex.  So the segment joining $z_{1}$ and $ z_{2}$ is contained  in $W^{+}_{\mathbb{H}}(PAP|_{V})$. Now we show that $W^{+}_{\mathbb{H}}(PAP|_{V}) \subseteq W^{+}_{\mathbb{H}}(A)$. Let $\omega \in W^{+}_{\mathbb{H}}(PAP|_{V})$,  then $\omega = \langle v, PAP|_{V}(v)\rangle $, for some $v\in V$, $\|v\|=1$. Since $P$ is orthogonal projection onto $V$, we have $\|Pv\| = \|v\|=1$ and 
	\begin{equation*}
	\omega = \langle Pv, APv\rangle_{\mathbb{H}} \in W^{+}_{\mathbb{H}}(A). 
	\end{equation*}
	This shows that the line segment joining $z_{1}$ and $z_{2}$ is contained in $W^{+}_{\mathbb{H}}(A)$. Hence $W^{+}_{\mathbb{H}}(A)$ is a convex subset of  $\mathbb{C}$.  
\end{proof}
\begin{note}
	Though the quaternionic numerical range of $A \in M_{n}(\mathbb{H})$ is not convex, from Theorem \ref{Theorem: TH} it is clear that every section of $W_{\mathbb{H}}(A)$ is convex.
\end{note}
\section{Numerical radius inequalities }
In this section we prove that the numerical radius of a quaternionic matrix $A\in M_{n}(\mathbb{H})$ is same as the numerical radius of complex matrix $\rchi_{A}\in M_{2n}(\mathbb{C})$. Further, we show that  the quaternionic numerical radius, denoted by ${\mathop{w}}_{\mathbb{H}}(\cdot)$ defines a norm on $M_{n}(\mathbb{H})$. In particular, for normal matrices over quaternions,  the numerical radius coinsides with the operator norm $\|\cdot \|$. We prove an inequality (see Theorem \ref{Theorem: betterestimation}) which is analogous to classical result that provides a better estimate for an upper bound of numerical radius.


We recall the definition for the numerical radius of matrices over quaternions: for a given $A \in M_{n}(\mathbb{H})$,
the quaternionic numerical radius, denoted by ${\mathop{w}}_{\mathbb{H}}(A)$, is defined as 
\begin{equation*}
{\mathop{w}}_{\mathbb{H}}(A):= \sup\Big\{|q|: q \in W_{\mathbb{H}}(A)\Big\}.
\end{equation*}

Since $W_{\mathbb{H}}(A)$ is circular, then the numerical radius of $A$ can also be defined as, 
\begin{equation*}
{\mathop{w}}_{\mathbb{H}}(A) = \sup \Big\{|z|:\; z \in W^{+}_{\mathbb{H}}(A)\Big\}.
\end{equation*}
Note that if $q \in W_{\mathbb{H}}(A)$, then $q = \langle X, AX \rangle_{\mathbb{H}} $ for some $X \in S_{\mathbb{H}^{n}}$. By Cauchy-Schwarz inequality, $|q|= |\langle X, AX\rangle_{\mathbb{H}}| \leq \|A\|$. This implies that
\begin{equation}\label{Equation:inequality}
{\mathop{w}}_{\mathbb{H}}(A) \leq \|A\|.
\end{equation}
Now we show that the quaternionic numerical radius of $A \in M_{n}(H)$ is same as the complex numerical radius, denoted by ${\mathop{w}}_{\mathbb{C}}(\rchi_{A})$, of $\rchi_{A}$.

\begin{thm}\label{Theorem: SameNumericalradius}
	Let $A\in M_{n}(\mathbb{H})$. Then 
	\begin{equation*}
	{\mathop{w}}_{\mathbb{H}}(A)= {\mathop{w}}_{\mathbb{C}}(\rchi_{A}) .
	\end{equation*}
\end{thm}
\begin{proof}
	Since $W_{\mathbb{H}}(A) \subseteq \Omega_{W_{\mathbb{C}}(\rchi_{A})}$, 
	we have
	\begin{equation*}
	{\mathop{w}}_{\mathbb{H}}(A) \leq \sup\big\{|q|:\; q \in \Omega_{W_{\mathbb{C}}(\rchi_{A})}\big\} = \sup\big\{|z|:\; z \in {W_{\mathbb{C}}(\rchi_{A})}\big\} = {\mathop{w}}_{\mathbb{C}}(\rchi_{A}).
	\end{equation*} 
	Now we prove the reverse inequality.  We know  from Proposition \ref{Proposition:projection}  that $W_{\mathbb{C}}(\rchi_{A}) =W_{\mathbb{H}}(A:\mathbb{C}) $. Thus   $z \in W_{\mathbb{C}}(\rchi_{A})$ if and only if there exist a  $z^{\prime} \in \mathbb{C}$ such that $z+z^{\prime}\cdot j \in W_{\mathbb{H}}(A)$ and $|z| \leq |z+z^{\prime}\cdot j|$. This implies that 
	\begin{align*}
	{\mathop{w}}_{\mathbb{C}}(\rchi_{A}) &= \sup \{|z|: z\in W_{\mathbb{C}}(\rchi_{A})\} \\
	&\leq \sup \{|z+z^{\prime}\cdot j|: \; z+z^{\prime} \cdot j \in W_{\mathbb{H}}(A)\} \\
	&= \sup \{|q|: \; q \in W_{\mathbb{H}}(A)\}\\
	&= {\mathop{w}}_{\mathbb{H}}(A).
	\end{align*}
	Hence that result.
\end{proof}
\begin{cor}\label{Corollary: NormNumerialradius}
	If $A \in M_{n}(\mathbb{H})$ is normal, then ${\mathop{w}}_{\mathbb{H}}(A) = \|A\|$.
\end{cor}
\begin{proof}Since $A$ is normal, then $\rchi_{A}$ is normal by Theorem \ref{Theorem: properties} and ${\mathop{w}}_{\mathbb{C}}(\rchi_{A}) = \|\rchi_{A}\|$ by \cite[Theorem 1.4-2]{GustafsonRao}. From Theorem \ref{Theorem: SameNumericalradius}, it is clear that
	\begin{equation*}
	{\mathop{w}}_{\mathbb{H}}(A) = {\mathop{w}}_{\mathbb{C}}(\rchi_{A}) = \|\rchi_{A}\|=\|A\|. \qedhere
	\end{equation*}
\end{proof}

\noindent Note that the same result for quaternionic normal operators is proved in \cite{Ramesh}. Now we show that ${\mathop{w}_{\mathbb{H}}}(A)$ is equivalent to the operator norm of $A$.
\begin{thm}\label{Theorem: Equivalent}
	Let $A \in M_{n}(\mathbb{H})$. Then  ${\mathop{w}_{\mathbb{H}}}(A)\leq \|A\| \leq 2{\mathop{w}_{\mathbb{H}}}(A)$. 
\end{thm}
\begin{proof} Since ${\mathop{w}}_{\mathbb{H}}(A) \leq \|A\|$, it is enough to prove the second inequality. By \cite[Theorem 1.3-1]{GustafsonRao}, we have 
	\begin{equation*}
	\|\rchi_{A}\| \leq 2 {\mathop{w}}_{\mathbb{C}}(\rchi_{A}).
	\end{equation*}
	Then by Theorem \ref{Theorem: SameNumericalradius}, we conclude that
	\begin{equation*}
	{\mathop{w}}_{\mathbb{H}}(A)\leq \|A\| \leq 2{\mathop{w}}_{\mathbb{H}}(A).
	\end{equation*}
	That is, the numerical radius is equivalent to operator norm $\|\cdot\|$. Hence ${\mathop{w}}_{\mathbb{H}}(\cdot)$ defines a norm on $M_{n}(\mathbb{H})$.
\end{proof}
\begin{rmk} We can use the technique follwed in the classical proof to show that   ${\mathop{w}}_{\mathbb{H}}(A)$ is equivalent to $\|A\|$ : 
	
	For every $X,Y \in \mathbb{H}^{n}$, by quaternionic version of polarization identity \cite[Equation (2.4)]{Ghiloni}, we have 
	\begin{equation*}
	4 \langle X, AY\rangle = \sum\limits_{\ell = 0}^{3} \Big[\big\langle Xe_{\ell}+Y, A(Xe_{\ell}+Y) \big\rangle_{\mathbb{H}} -  \big\langle Xe_{\ell}-Y, A(Xe_{\ell}-Y) \big\rangle_{\mathbb{H}}\Big]e_{\ell},
	\end{equation*}
	where $e_{0}=1, e_{1}= i, e_{2}= j$ and $e_{3}=k$. By taking modulus on both sides, we get 
	\begin{align*}
	4 |\langle X, AY\rangle| &\leq {\mathop{w}}_{\mathbb{H}}(A) \sum\limits_{\ell = 0}^{3} \Big[\|Xe_{\ell}+Y\|^{2}+ \|Xe_{\ell}-Y\|^{2}\Big]\\
	&\leq 8{\mathop{w}}_{\mathbb{H}}(A) (\|X\|^{2}+\|Y\|^{2}).
	\end{align*}
	By taking supremum over $X, Y \in S_{\mathbb{H}^{n}}$ on both sides, we get
	\begin{equation*}
	\|A\| \leq 4 \; {\mathop{w}}_{\mathbb{H}}(A). \qedhere
	\end{equation*}
	But the technique followed in  Theorem \ref{Theorem: Equivalent} provides a better estimation for the lower bound of ${\mathop{w}}_{\mathbb{H}}(A)$.
\end{rmk}
\begin{thm}
	Let $A, B \in M_{n}(\mathbb{H})$. Then 
	\begin{enumerate}
		\item ${\mathop{w}}_{\mathbb{H}}(ABA^{\ast}) \leq \|A\|^{2} {\mathop{w}}_{\mathbb{H}}(B)$. In particular, for a {\bf compression}  $PBP$ of $B$, where $P \in M_{n}(\mathbb{H})$ be such that $P^{\ast}=P$ and $P^{2}=P$, we have
		\begin{equation*}
		{\mathop{w}}_{\mathbb{H}}(PBP) \leq {\mathop{w}}_{\mathbb{H}}(B).
		\end{equation*}
		\item ${\mathop{w}}_{\mathbb{H}}\big(\begin{bmatrix}
		A &0\\
		0& B
		\end{bmatrix}\big) = \max\{{\mathop{w}}_{\mathbb{H}}(A),\; {\mathop{w}}_{\mathbb{H}}(B)\}$.
	\end{enumerate}
	
\end{thm}
\begin{proof}
	Proof of $(1):$	Let $X \in S_{\mathbb{H}^{n}}$, then 
	\begin{align*}
	|\langle X, ABA^{\ast}X\rangle_{\mathbb{H}}| &= |\langle A^{\ast}X, BA^{\ast}X \rangle_{\mathbb{H}} | \\
	&	\leq \|A^{\ast}X\|^{2} {\mathop{w}}_{\mathbb{H}}(B)\\
	&\leq \|A^{\ast}\|^{2}{\mathop{w}}_{\mathbb{H}}(A)\\
	&= \|A\|^{2}{\mathop{w}}_{\mathbb{H}}(A).
	\end{align*}
	Hence ${\mathop{w}}_{\mathbb{H}}(ABA^{\ast}) \leq \|A\|^{2} {\mathop{w}}_{\mathbb{H}}(B)$. In particular, for a compression of $B$, 
	\begin{equation*}
	{\mathop{w}}_{\mathbb{H}}(PBP) \leq \|P\| {\mathop{w}}_{\mathbb{H}}(B)= {\mathop{w}}_{\mathbb{H}}(B).
	\end{equation*}
	\noindent Proof of $(2):$
	Let $\begin{bmatrix}
	X\\
	Y
	\end{bmatrix} \in S_{\mathbb{H}^{n}\oplus \mathbb{H}^{n}}$, then 
	\begin{align*}
	|\Big\langle \begin{bmatrix}
	X\\
	Y
	\end{bmatrix}, \begin{bmatrix}
	A &0\\
	0& B
	\end{bmatrix} \begin{bmatrix}
	X\\
	Y
	\end{bmatrix}\Big\rangle_{\mathbb{H}}| &\leq |\langle X, AX\rangle_{\mathbb{H}}| + |\langle Y, BY\rangle_{\mathbb{H}}|\\
	&\leq {\mathop{w}}_{\mathbb{H}}(A)\|X\|^{2}+ {\mathop{w}}_{\mathbb{H}}(B)\|Y\|^{2}\\
	&\leq \max\{{\mathop{w}}_{\mathbb{H}}(A),{\mathop{w}}_{\mathbb{H}}(B)\}(\|X\|^{2}+\|Y\|^{2})\\
	&= \max\{{\mathop{w}}_{\mathbb{H}}(A),{\mathop{w}}_{\mathbb{H}}(B)\}.
	\end{align*}
	Thus ${\mathop{w}}_{\mathbb{H}}\big(\begin{bmatrix}
	A &0\\
	0& B
	\end{bmatrix}\big) \leq \max\{{\mathop{w}}_{\mathbb{H}}(A),{\mathop{w}}_{\mathbb{H}}(B)\}$.  
	We show the reverse inequality. Since $\langle X, AX \rangle_{\mathbb{H}} = \Big\langle \begin{bmatrix}
	X\\
	0
	\end{bmatrix}, \begin{bmatrix}
	A &0\\
	0& B
	\end{bmatrix} \begin{bmatrix}
	X\\
	0
	\end{bmatrix}\Big\rangle_{\mathbb{H}} $ for every $X \in S_{\mathbb{H}^{n}}$, it implies that ${\mathop{w}}_{\mathbb{H}}(A) \leq {\mathop{w}}_{\mathbb{H}}\big(\begin{bmatrix}
	A &0\\
	0& B
	\end{bmatrix}\big)$. Similarly,   ${\mathop{w}}_{\mathbb{H}}(B) \leq {\mathop{w}}_{\mathbb{H}}\big(\begin{bmatrix}
	A &0\\
	0& B
	\end{bmatrix}\big) $. Therefore,
	\begin{equation*}
	\max\{{\mathop{w}}_{\mathbb{H}}(A),\; {\mathop{w}}_{\mathbb{H}}(B)\} \leq {\mathop{w}}_{\mathbb{H}}\big(\begin{bmatrix}
	A &0\\
	0& B
	\end{bmatrix}\big).  \qedhere
	\end{equation*}
\end{proof}
\begin{note}
	In case of the associated complex matrix $\rchi_{A}$ of $A \in M_{n}(\mathbb{H})$, the better estimate for an upper bound of ${\mathop{w}}_{\mathbb{C}}(\rchi_{A})$  is given in  \cite{Kittaneh} as follows: 
	\begin{equation}\label{Equation: complexgeneralization}
	{\mathop{w}}_{\mathbb{C}}(\rchi_{A}) \leq \frac{1}{2}\Big( \|\rchi_{A}\|+ \|\rchi_{A}^{2}\|^{\frac{1}{2}}\Big). 
	\end{equation}
\end{note}
The same result is true for matrices over quaternions. We prove the following theorem. 
\begin{thm} \label{Theorem: betterestimation}
	If $A \in M_{n}(\mathbb{H})$, then 
	\begin{equation*}
	{\mathop{w}}_{\mathbb{H}}(A) \leq \frac{1}{2}\Big( \|A\|+ \|A^{2}\|^{\frac{1}{2}}\Big).
	\end{equation*}
\end{thm}
\begin{proof}
	By Theorem \ref{Theorem: SameNumericalradius} and Equation (\ref{Equation: complexgeneralization}), we have that 
	\begin{align*}
	{\mathop{w}}_{\mathbb{H}}(A) = {\mathop{w}}_{\mathbb{C}}(\rchi_{A}) &\leq \frac{1}{2} \Big( \|\rchi_{A}\|+ \|\rchi_{A}^{2}\|^{\frac{1}{2}}\Big)\\
	&=  \frac{1}{2}\Big( \|\rchi_{A}\|+ \|\rchi_{A^{2}}\|^{\frac{1}{2}}\Big)\\
	&= \frac{1}{2}\Big( \|A\|+ \|A^{2}\|^{\frac{1}{2}}\Big). \qedhere
	\end{align*}
\end{proof}
As a consequence of Theorem \ref{Theorem: betterestimation}, under certain assumptions, we show that the lower and the upper bounds given in  Theorem \ref{Theorem: Equivalent} coinsides with the numerical radius.
\begin{cor}\label{Corollary: final}
	Let $A \in M_{n}(\mathbb{H})$. Then the following assertions hold true:
	\begin{enumerate}
		\item If $A^{2}=0$, then ${\mathop{w}}_{\mathbb{H}}(A) = \frac{1}{2} \|A\|$.
		\item If ${\mathop{w}}_{\mathbb{H}}(A) = \|A\|$, then $\|A\|^{2} = \|A^{2}\|$.
	\end{enumerate}
\end{cor}
\begin{proof}
	Proof of $(1)$: If $A^{2}=0$, then by Theorem \ref{Theorem: betterestimation}, we have 
	\begin{equation*}
	\frac{1}{2} \|A\| \leq {\mathop{w}}_{\mathbb{H}}(A) \leq \frac{1}{2} \|A\|.
	\end{equation*}
	We get ${\mathop{w}}_{\mathbb{H}}(A) = \frac{1}{2}\|A\|$.

	\noindent Proof of $(2)$: Since $M_{n}(\mathbb{H})$ is a normed algebra with the operator norm, we see that
	\begin{equation*}
	\|A^{2}\| \leq \|A\| \|A\| = \|A\|^{2}.
	\end{equation*}
	Now we prove the reverse inequality. Since ${\mathop{w}}_{\mathbb{H}}(A) = \|A\|$ and by Theorem \ref{Theorem: betterestimation}, it follows that
	\begin{equation*}
	2\|A\| \leq   \|A\| + \|A^{2}\|^{\frac{1}{2}},
	\end{equation*}
	i.e., $\|A\|^{2} \leq \|A^{2}\|$. Therefore, $\|A\|^{2} = \|A^{2}\|$. \qedhere
\end{proof}
Note that the converse of $(1)$ and $(2)$ of Corollary \ref{Corollary: final} is not true for $n >2$. We provide example for each one of them. 

\begin{eg} Converse of $(1)$ of Corollary \ref{Corollary: final} is not true for $n>2$:
	
	Let 
	\begin{equation*}
	A = \begin{bmatrix}
	0 & 1+\sqrt{3}k & 0\\
	0&0&0\\
	0&0&j
	\end{bmatrix}_{3\times 3}.
	\end{equation*}
	We get 
	\begin{equation*}
	A^{\ast}
	A = \begin{bmatrix}
	0 & 0 & 0\\
	1-\sqrt{3}k&0&0\\
	0&0&-j
	\end{bmatrix} 
	\begin{bmatrix}
	0 & 1+\sqrt{3}k & 0\\
	0&0&0\\
	0&0&j
	\end{bmatrix} = \begin{bmatrix}
	0 &0&0\\
	0&4&0\\
	0&0&1
	\end{bmatrix}
	\end{equation*} 
	Then 
	\begin{equation*}
	\|A\| = \sqrt{\|A^{\ast}A\|} = \sqrt{4}=2.
	\end{equation*}
	Let $X: = \begin{bmatrix}
	a\\
	b\\
	c
	\end{bmatrix} \in S_{\mathbb{H}^{3}}$.
	Then
	\begin{align*}
	|\langle X, A X\rangle_{\mathbb{H}}| &= |\;\overline{a} (1+\sqrt{3}k)b + \overline{c}jc\;|\\
	&\leq 2 |\overline{a}b|+ |c|^{2}\\
	& \leq 2 \big(\frac{|a|^{2}+|b|^{2}}{2}\big) + |c|^{2} \;\;  (\text{by Young's inequality})\\
	&=1. 
	\end{align*}
	If we choose $a = b =0$ and $c=1$, then $|\langle X, A X \rangle_{\mathbb{H}}|=1$. This shows that ${\mathop{w}}_{\mathbb{H}}(A) = 1 = \frac{1}{2}\|A\|$, but $A^{2}\neq 0$.
\end{eg}
\begin{eg} Converse of $(2)$ of Corollary \ref{Corollary: final} is not true for $n >2$:
	
	Let
	\begin{equation*}
	A = \begin{bmatrix}
	0&0&0\\
	j&0&0\\
	0&k&0
	\end{bmatrix}_{3\times 3}.
	\end{equation*}
	Then $\|A^{2}\| = \|A\|^{2}=1$, but ${\mathop{w}}_{\mathbb{H}}(A) = \frac{1}{\sqrt{2}} < 1$.
\end{eg}
\begin{rmk}
	If $n =2$, then the converse of $(1)$ and $(2)$ of Corollary \ref{Corollary: final} holds true. The proof follows similar lines as in complex case (see \cite{Kittaneh} and references therein). 
\end{rmk}

\section*{Acknowledgment}
The author is  thankful to the  Department of Atomic Energy (DAE), Government of India for financial support and ISI Bangalore for providing necessary facilities to carry out this work. We thank Prof.  B.V.R. Bhat for valuable suggestions.



\begin{thebibliography}{99}
	\bibitem{Au-Yeung1}Y.-H. Au-Yeung, On the eigenvalues and numerical range of a quaternionic matrix,  World Sci. Publ., River Edge, NJ.
	\bibitem{Au-Yeung2} Y. H. Au-Yeung, On the convexity of numerical range in quaternionic Hilbert spaces, Linear and Multilinear Algebra {\bf 16} (1984).
	\bibitem{Ghiloni} R. Ghiloni, V. Moretti\ and\ A. Perotti, Continuous slice functional calculus in quaternionic Hilbert spaces, Rev. Math. Phys. {\bf 25} (2013)
	\bibitem{GustafsonRao}K. E. Gustafson\ and\ D. K. M. Rao, {\it Numerical range}, Universitext, Springer-Verlag, New York, 1997.
	\bibitem{Jamison} J. E. Jamison, Numerical range and numerical radius in quaternionic Hilbert spaces, Doctoral Dissertation, Univ. of Missouri, 1972.
	\bibitem{Kittaneh} F. Kittaneh, A numerical radius inequality and an estimate for the numerical radius of the Frobenius companion matrix, Studia Math. {\bf 158} (2003), no.~1, 11--17.
	\bibitem{Ramesh} G. Ramesh, On the numerical radius of a quaternionic normal operator, Adv. Oper. Theory {\bf 2} (2017).
	\bibitem{Santhosh} G. Ramesh\ and\ P. Santhosh Kumar, On the polar decomposition of right linear operators in quaternionic Hilbert spaces, J. Math. Phys. {\bf 57} (2016). 
	\bibitem{Thompson} R. C. Thompson, The upper numerical range of a quaternionic matrix is not a complex numerical range, Linear Algebra Appl. {\bf 254} (1997).
	\bibitem{So1} W. So, R. C. Thompson\ and\ F. Z. Zhang, The numerical range of normal matrices with quaternion entries, Linear and Multilinear Algebra {\bf 37} (1994).
	\bibitem{So2}W. So\ and\ R. C. Thompson, Convexity of the upper complex plane part of the numerical range of a quaternionic matrix, Linear and Multilinear Algebra {\bf 41} (1996).
	\bibitem{Zhang} F. Zhang, Quaternions and matrices of quaternions, Linear Algebra Appl. {\bf 251} (1997), 21--57.
\end{thebibliography}
\end{document}